\documentclass{amsart}
\usepackage{graphicx}
\newtheorem{thm}{Theorem}[section]
\newtheorem{cor}[thm]{Corollary}
\newtheorem{lem}[thm]{Lemma}
\newtheorem{prop}[thm]{Proposition}
\theoremstyle{definition}
\newtheorem{defn}[thm]{Definition}
\theoremstyle{example}

\theoremstyle{remark}
\newtheorem{rem}[thm]{Remark}
\numberwithin{equation}{section}

\begin{document}

\title[An appropriate representation space for controlled g-frames]
{An appropriate representation space for controlled g-frames}

\author[M. Forughi]{$^a$ Maryam Forughi}
\address{$^a$ Islamic Azad University of Shabester\\Tabriz-Shabestar\\
Iran} \email{Maryam.forughi@yahoo.com}

\author[E. Osgooei]{$^b$ Elnaz Osgooei $^*$}\footnote{Corresponding author $^*$}
\address{$^b$ Faculty of Science\\Urmia University of Technology\\Urmia\\
Iran} \email{e.osgooei@uut.ac.ir}

\author[A. Rahimi]{$^c$ Asghar Rahimi}
\address{$^c$ Department of Mathematics\\ University of Maragheh\\ Iran\\}
\email{rahimi@maragheh.ac.ir}

\author[M. Javahernia]{$^d$ Mojgan Javahernia}
\address{$^d$ Islamic Azad University of Shabester\\Tabriz-Shabestar\\
Iran} \email{javahernia\_math@yahoo.com}
\dedicatory{}

\subjclass[2010]{Primary 42C15; Secondary 46C99, 41A58}

\keywords{Controlled g-frame, controlled g-dual frame, trace class
operator.}

\begin{abstract}
In this paper, motivating the range of operators, we propose an
appropriate representation space to introduce synthesis and analysis
operators of controlled g-frames and discuss the properties of these
operators. Especially, we show that the operator obtained by the
composition of the synthesis and analysis operators of two
controlled g-Bessel sequence is a trace class operator. Also, we
define the canonical controlled g-dual and show that this dual gives
rise to expand coefficients with the minimal norm. Finally, we
extend some known equalities and inequalities for controlled
g-frames.
\end{abstract}
\maketitle
\section{Introduction and Preliminaries}
Frames were first introduced in the context of non-harmonic Fourier
series by Duffin and Schaeffer \cite{Duffin}. During the last $20$s
the theory of frames has been developed rapidly and because of the
abundant use of frames in engineering and applied sciences, many
generalization of frames have come into play.
\\G-frames that include the concept of ordinary frames have been introduced
by Sun \cite{sun} and improved by many authors \cite{Hua, Don, naf,
Rah}. Controlled frames have been improved recently to improve the
numerical efficiency of interactive algorithms that inverts the
frame operator \cite{Bal}. Following that, controlled frames have
been generalized to another kinds of frames \cite{Hua, Kho, Mus,
osg, Rah, Najaf, Don}.
\\In this paper, motivating the concept of g-frames and controlled
frames we define controlled g-frames. In Section 2, imagined the
range of an operator, a new representation space is introduced such
that the synthesis and analysis operators could be defined. In
Section 3, controlled g- dual frames and canonical controlled g-dual
frames are introduced and shown  that canonical g-dual gives rise to
expand coefficients with the minimal norm. Finally, some equalities
and inequalities are presented for controlled g-frames and
especially for their operators in Section 4.
\\Throughout this paper, $H$ is a  separable Hilbert space,
$\{H_i\}_{i\in\Bbb I}$ is the collection of Hilbert spaces,
$\mathcal{B}(H,K)$ is the family of all linear bounded operators
from $H$ into $K$ and $\mathcal{GL}(H)$ is the set of all bounded
linear operators which have bounded inverses.
\\At first, we collect some definitions and basic results that are needed in
the paper.
\begin{lem}(\cite{naj})\label{l0}
Let $u\in\mathcal{B}(H)$ be a self-adjoint operator and
$v:=au^2+bu+c$ where $a,b,c\in\Bbb R$.
\begin{enumerate}
\item[(i)] If $a>0$, then
$$\inf_{\Vert f\Vert=1}\langle vf, f\rangle\geq\frac{4ac-b^2}{4a}.$$
\item[(ii)] If $a<0$, then
$$\sup_{\Vert f\Vert=1}\langle vf, f\rangle\leq\frac{4ac-b^2}{4a}.$$
\end{enumerate}
\end{lem}
\begin{lem}\label{l2}(\cite{cas1})
If $u,v$ are operators on $H$ satisfying $u+v=id_{H}$, then
$$u-v=u^2-v^2.$$
\end{lem}
If an operator $ u$ has closed range, then there exists a
right-inverse operator $u^ \dagger$ (pseudo-inverse of $u$) in the
following senses (see \cite{ch}).
\begin{lem}\label{l1}
Let $u\in\mathcal{B}(K,H)$  be a bounded operator with closed range
$\mathcal{R}_{u}$. Then there exists a bounded operator $u^\dagger
\in\mathcal{B}(H,K)$ for which
$$uu^{\dagger} x=x, \ \ x\in \mathcal{R}(u).$$
\end{lem}
\begin{defn}(\textbf{g-frame})
A family $\Lambda:=\lbrace
\Lambda_i\in\mathcal{B}(H,H_i)\rbrace_{i\in\Bbb I}$ is called a
g-frame for $H$ with respect to $\lbrace H_i\rbrace_{i\in\Bbb I}$,
if there exist $0<A\leq B<\infty$ such that
\begin{equation} \label{first}
A\Vert f\Vert^2\leq\sum_{i\in\Bbb I}\Vert \Lambda_{i}f\Vert^2\leq
B\Vert f\Vert^2, \ \  f\in H.
\end{equation}
If only the second inequality in (\ref{first}) satisfy, then we say
that $\{\Lambda_{i}\}_{i\in\Bbb I}$ is a g-Bessel sequence with
upper bound $B$.
\end{defn}
If $\Lambda$ is a g-Bessel sequence, then the synthesis and analysis
operators are defined by
\begin{align*}
T_{\Lambda}: &(\sum_{i\in\Bbb I}\oplus H_i)_{\ell^{2}} \rightarrow H , \qquad T_{\Lambda}^{*}: H \rightarrow (\sum_{i\in\Bbb I}\oplus H_i)_{\ell^{2}},\\
T_{\Lambda}(\lbrace f_i&\rbrace_{i\in\Bbb I})=\sum_{i\in\Bbb
I}\Lambda_{i}^{\ast}(f_i) , \qquad T_{\Lambda}^{\ast}(f)=\lbrace
\Lambda_{i}f\rbrace_{i\in\Bbb I},
\end{align*}
where
\begin{eqnarray*}
(\sum_{i\in\Bbb I}\oplus H_i)_{\ell^{2}}=\big\lbrace\lbrace
f_i\rbrace_{i\in\Bbb I} \ \vert \ f_i\in H_i \ , \ \sum_{i\in\Bbb
I}\Vert f_i\Vert^2<\infty \big\rbrace,
\end{eqnarray*}
and, the g-frame operator is given by
$$S_{\Lambda}f=T_{\Lambda}T^*_{\Lambda}f=\sum_{i\in\Bbb I}\Lambda^{\ast}_{i}\Lambda_{i}f, \ \ f\in H,$$
which is positive, self-adjoint and invertible (see \cite{sun}).
\section{Controlled g-frames and their operators}
Controlled frames for spherical wavelets were introduced  in
\cite{bog} to get a numerically more efficient approximation
algorithm. In this section by extending the concept of controlled
frames and g-frames, we define the concept of controlled g-frames
and construct an appropriate representation space to organize the
synthesis and analysis operators.
\begin{defn}\cite{Bal}
Let $C\in\mathcal{GL}(H)$. We say that $F:=\{ f_i\}_{i\in\Bbb I}$ is
a $C$-controlled frame for $H$ if there exist $0<A_{C}\leq
B_{C}<\infty$ such that for each $f\in H$
\begin{equation}\label{cont}
A_{C}\Vert f\Vert^2\leq\sum_{i\in\Bbb I}\langle f, f_i\rangle\langle
Cf_i, f\rangle\leq B_{C}\Vert f\Vert^2.
\end{equation}
\end{defn}
\begin{defn}
Let $C,C'\in\mathcal{GL}(H)$. We say that
$\Lambda:=\{\Lambda_i\in\mathcal{B}(H, H_i)\}_{i\in\Bbb I}$  is a
$(C,C')$-controlled g-frame for $H$ if there exist $0<A_{CC'}\leq
B_{CC'}<\infty$ such that for each $f\in H$
\begin{equation}\label{cont}
A_{CC'}\Vert f\Vert^2\leq\sum_{i\in\Bbb I}\langle \Lambda_i C'f,
\Lambda_i Cf\rangle\leq B_{CC'}\Vert f\Vert^2.
\end{equation}
\end{defn}
For simplicity, we use a notation $CC'$ instead of $(C,C')$. We call
$\Lambda$ a \textit{Parseval $CC'$-controlled g-frame} if
$A_{CC'}=B_{CC'}=1$. When the right hand inequality of \eqref{cont}
holds, then $\Lambda$ is called a \textit{$CC'$-controlled g-Bessel
sequence} for $H$ with bound $B_{C}$. \\If $\Lambda$ is a
$CC'$-controlled g-frame for $H$ and $C^*\Lambda^*_i\Lambda_iC'$ is
positive for each $i\in\Bbb I$, then we have
$$A_{CC'}\Vert f\Vert^2\leq\sum_{i\in\Bbb I}\Vert(C^*\Lambda^*_i\Lambda_iC')^
{\frac{1}{2}}f\Vert^2\leq B_{CC'}\Vert f\Vert^2,\ \ f\in H.$$
Consider a proper representation space by
$$K:=\big\lbrace\{(C^*\Lambda^*_i\Lambda_iC')^
{\frac{1}{2}}f\}_{i\in\Bbb I} : \ f\in H\big\rbrace\subset
(\sum_{i\in\Bbb I}\oplus H_i)_{\ell^{2}}.$$ Since $K$ is a closed
subspace of $(\sum_{i\in\Bbb I}\oplus H_i)_{\ell^{2}}$, we can
define the synthesis and analysis operators of $CC'$-controlled
g-frames by
\begin{align*}
T_{CC'}:K\longrightarrow H,\\
T_{CC'}(\lbrace
(C^*\Lambda^*_i\Lambda_iC')^{\frac{1}{2}}f\rbrace_{i\in\Bbb
I})&=\sum_{i\in\Bbb I}C^*\Lambda^*_i\Lambda_i C'f
\end{align*}
and
\begin{align*}
T_{CC'}^*&:H\longrightarrow K,\\
T_{CC'}^*(f)=&\lbrace (C^*\Lambda^*_i\Lambda_iC')^{\frac{1}{2}}f
\rbrace_{i\in\Bbb I}.
\end{align*}
Thus, the $CC'$-controlled g-frame operator is given by
$$S_{CC'}f=T_{CC'}T^*_{CC'}f=\sum_{i\in\Bbb I}C^*\Lambda^*_i\Lambda_i C'f, \qquad f\in H.$$
So,
$$\langle S_{CC'}f, f\rangle=\sum_{i\in\Bbb I}\langle \Lambda_i C'f, \Lambda_i Cf\rangle, \qquad f\in H,$$
and
$$A_{CC'}Id_H\leq S_{CC'}\leq B_{CC'}Id_H.$$
Therefore, $S_{CC'}$ is a positive, self-adjoint and invertible
operator (see \cite{Rah} and \cite{Mus}). Thus, since
$f=S_{CC'}S^{-1}_{CC'}f=S^{-1}_{CC'}S_{CC'}f$, we have
\begin{equation}\label{s1}
f=\sum_{i\in\Bbb I}C^*\Lambda^*_i\Lambda_i
C'S^{-1}_{CC'}f=\sum_{i\in\Bbb I}S^{-1}_{CC'}C^*\Lambda^*_i\Lambda_i
C'f,\ \ \ f\in H.
\end{equation}
\begin{rem}\label{exam}
We introduce a Parseval $CC'$-controlled g-frame for $H$ by the
$CC'$-controlled g-frame operator. Suppose that $\Lambda$ is a
$CC'$-controlled g-frame for $H$. Since $S_{CC'}$(or $S_{CC'}^{-1}$)
is positive in $\mathcal{B}(H)$ and $\mathcal{B}(H)$ is a
$C^*$-algebra, then there exists a unique positive square root
$S^{\frac{1}{2}}_{CC'}$ (or $S^{-\frac{1}{2}}_{CC'}$) which commutes
with $S_{CC'}$ and $S_{CC'}^{-1}$. Therefore, for any $f\in H$ we
can write
$$f=S^{-\frac{1}{2}}_{CC'}S_{CC'}S^{-\frac{1}{2}}_{CC'}f=\sum_{i\in\Bbb I}S^{-\frac{1}{2}}_{CC'}C^*\Lambda^*_i\Lambda_iC'S^{-\frac{1}{2}}_{CC'}f.$$
Now, assume that $S^{-\frac{1}{2}}_{CC'}$ commutes with $C, C'$.
Then we get
$$\Vert f\Vert^2=\langle f, f\rangle=\sum_{i\in\Bbb I}\langle\Lambda_iS^{-\frac{1}{2}}_{CC'}C'f,
\Lambda_iS^{-\frac{1}{2}}_{CC'}Cf\rangle.$$ Hence,
$\{\Lambda_iS^{-\frac{1}{2}}_{CC'}\}_{i\in\Bbb I}$ is a Parseval
$CC'$-controlled g-frame for $H$.
\end{rem}
\begin{thm}\label{th1}
A sequence $\Lambda$ is a $CC'$-controlled g-Bessel sequence for $H$
with bound $B_{CC'}$ if and only if the operator
\begin{align*}
T_{CC'}:K\longrightarrow H,\\
T_{CC'}(\lbrace
(C^*\Lambda^*_i\Lambda_iC')^{\frac{1}{2}}f\rbrace_{i\in\Bbb
I})&=\sum_{i\in\Bbb I}C^*\Lambda^*_i\Lambda_i C'f
\end{align*}
is a well-defined and bounded operator with $\Vert
T_{CC'}\Vert\leq\sqrt{B_{CC'}}$.
\end{thm}
\begin{proof}
We only need to prove the sufficient condition. Let $T_{CC'}$ be a
well-defined and bounded operator with $\Vert
T_{CC'}\Vert\leq\sqrt{B_{CC'}}$. For each $f\in H$ we have
\begin{align*}
\sum_{i\in\Bbb I}\langle \Lambda_i C'f, \Lambda_i Cf\rangle&=\sum_{i\in\Bbb I}\langle C^*\Lambda^*_i\Lambda_i C'f, f\rangle\\
&=\big\langle T_{CC'}(\lbrace (C^*\Lambda^*_i\Lambda_iC')^{\frac{1}{2}}f\rbrace_{i\in\Bbb I}), f\big\rangle\\
&\leq\Vert T_{CC'}\Vert
\Vert\lbrace(C^*\Lambda^*_i\Lambda_iC')^{\frac{1}{2}}f\rbrace_{i\in\Bbb
I}\Vert\Vert f\Vert.
\end{align*}
But
$$\Vert\lbrace(C^*\Lambda^*_i\Lambda_iC')^{\frac{1}{2}}f\rbrace_{i\in\Bbb I}\Vert^2=\sum_{i\in\Bbb I}\langle \Lambda_i C'f, \Lambda_i Cf\rangle.$$
It follows that
$$\sum_{i\in\Bbb I}\langle \Lambda_i C'f, \Lambda_i Cf\rangle\leq B_{CC'}\Vert f\Vert^2,$$
and this means that $\Lambda$ is a $CC'$-controlled g-Bessel
sequence.
\end{proof}
\begin{thm}
A sequence $\Lambda$  is a $CC'$-controlled g-frame for $H$ if and
only if
$$T_{CC'}:(\lbrace (C^*\Lambda^*_i\Lambda_iC')^{\frac{1}{2}}f\rbrace_{i\in\Bbb I})\longmapsto\sum_{i\in\Bbb I}C^*\Lambda^*_i\Lambda_i C'f$$
is a well-defined and surjective operator.
\end{thm}
\begin{proof}
First, suppose that $\Lambda$ is a $CC'$-controlled g-frame for $H$.
Since, $S_{CC'}$ is a surjective operator, so $T_{CC'}$. For the
opposite implication, by Theorem \ref{th1}, $T_{CC'}$ is a
well-defined and bounded operator. So, $\Lambda$ is a
$CC'$-controlled g-Bessel sequence. Now, for each $f\in H$, we have
$f=T_{CC'}T^{\dagger}_{CC'}f$. Hence,
\begin{align*}
\Vert f\Vert^4&=\vert\langle f, f\rangle\vert^2\\
&=\vert\langle T_{CC'}T^{\dagger}_{CC'}f, f\rangle\vert^2\\
&=\vert\langle T^{\dagger}_{CC'}f, T^*_{CC'}f\rangle\vert^2\\
&\leq\Vert T^{\dagger}_{CC'}\Vert^2\Vert f\Vert^2\sum_{i\in\Bbb
I}\langle \Lambda_i C'f, \Lambda_i Cf\rangle.
\end{align*}
We conclude that
$$\frac{1}{\Vert T^{\dagger}_{CC'}\Vert^2}\Vert f\Vert^2\leq\sum_{i\in\Bbb I}\langle \Lambda_i C'f, \Lambda_i Cf\rangle, \quad f\in H.$$
\end{proof}
\begin{thm}
Let $\Lambda=\{\Lambda_i\in\mathcal{B}(H, H_i)\}_{i\in\Bbb I}$ and
$\Theta:=\{\Theta_i\in\mathcal{B}(H, H_i)\}_{i\in\Bbb I}$ be two
$CC'$-controlled g-Bessel sequence for $H$ with bounds $B_1$ and
$B_2$, respectively. If $T_{\Lambda}$ and $T_{\Theta}$ are their
synthesis operators such that $T_{\Lambda}T^*_{\Theta}=Id_H$, then
both $\Lambda$ and $\Theta$ are $CC'$-controlled g-frames for $H$.
\end{thm}
\begin{proof}
For each $f\in H$, we have
\begin{align*}
\Vert f\Vert^4&=\langle f, f\rangle^2\\
&=\langle T^*_{\Lambda}f, T^*_{\Theta}f\rangle^2\\
&\leq\Vert T^*_{\Lambda}f\Vert^2\Vert T^*_{\Theta}f\Vert^2\\
&=\big(\sum_{i\in\Bbb I}\langle \Lambda_i C'f, \Lambda_i Cf\rangle\big)\big(\sum_{i\in\Bbb I}\langle \Theta_i C'f, \Theta_i Cf\rangle\big)\\
&\leq\big(\sum_{i\in\Bbb I}\langle \Lambda_i C'f, \Lambda_i
Cf\rangle\big) B_2\Vert f\Vert^2.
\end{align*}
Hence,
$$B_2^{-1}\Vert f\Vert^2\leq\big(\sum_{i\in\Bbb I}\langle \Lambda_i C'f, \Lambda_i Cf\rangle\big),$$
and $\Lambda$ is a $CC'$-controlled g-frame. Similarly, $\Theta$ is
a $CC'$-controlled g-frame with lower bound $B^{-1}_1$.
\end{proof}
\begin{thm}
Let $\Lambda=\{\Lambda_i\in\mathcal{B}(H, H_i)\}_{i\in\Bbb I}$ and
$\Theta:=\{\Theta_i\in\mathcal{B}(H, H_i)\}_{i\in\Bbb I}$ be two
$CC'$-controlled g-Bessel sequence for $H$ with bounds $B_1$ and
$B_2$, respectively where $\vert\Bbb I\vert<\infty$. If
$\Phi:=T_{\Lambda}T^*_{\Theta}$, then $\Phi$ is a trace class
operator.
\end{thm}
\begin{proof}
Suppose that $\Phi=u\vert\Phi\vert$ is the polar decomposition of
$\Phi$ where $u\in\mathcal{B}(H)$ is a partial isometry. So,
$\vert\Phi\vert=u^*\Phi$. Let $\{e_i\}_{i\in\Bbb I}$ be an
orthonormal basis for $H$. We have
\begin{align*}
\mathop{\rm tr}(\vert\Phi\vert)&=\sum_{i\in\Bbb I}\langle\vert\Phi\vert e_i, e_i\rangle\\
&=\sum_{i\in\Bbb I}\langle T^*_{\Theta}e_i, T^*_{\Lambda}ue_i\rangle\\
&=\sum_{i\in\Bbb I}\langle\lbrace (C^*\Theta^*_j\Theta_j C')^{\frac{1}{2}}e_i \rbrace_{j\in\Bbb I},\lbrace (C^*\Lambda^*_j\Lambda_j C')^{\frac{1}{2}}ue_i \rbrace_{j\in\Bbb I}\rangle\\
&\leq\sum_{i\in\Bbb I}\sum_{j\in\Bbb I}\Vert(C^*\Theta^*_j\Theta_j C')^{\frac{1}{2}}e_i\Vert \Vert(C^*\Lambda^*_j\Lambda_j C')^{\frac{1}{2}}ue_i\Vert\\
&\leq\sum_{i\in\Bbb I}\big(\sum_{j\in\Bbb I}\Vert(C^*\Theta^*_j\Theta_j C')^{\frac{1}{2}}e_i\Vert^2\big)^{\frac{1}{2}}\big(\sum_{j\in\Bbb I}\Vert(C^*\Lambda^*_j\Lambda_j C')^{\frac{1}{2}}ue_i\Vert^2\big)^{\frac{1}{2}}\\
&\leq\sum_{i\in\Bbb I}\sqrt{B_1 B_2}\Vert ue_i\Vert <\infty.
\end{align*}
\end{proof}
 \section{Controlled g-dual frames}
 In this section by considering that $C=C'$ and $\Lambda$ is a
$C^2$-controlled g-frame for $H$, we define a canonical controlled
g-dual and show that this canonical dual is a g-frame and gives rise
to expand coefficients with the minimal norm.
\begin{defn}
Suppose that $\Lambda=\{\Lambda_i\in\mathcal{B}(H, H_i)\}_{i\in\Bbb
I}$ and
$\widetilde{\Lambda}:=\{\widetilde{\Lambda}_i\in\mathcal{B}(H,
H_i)\}_{i\in\Bbb I}$ are two $CC'$-controlled g-Bessel sequence for
$H$ with  synthesis operators $T_{\Lambda}$ and
$T_{\widetilde{\Lambda}}$, respectively. We say that
$\widetilde{\Lambda}$ is a $CC'$-controlled g-dual of $\Lambda$ if
$$T_{\Lambda} T^*_{\widetilde{\Lambda}}=Id_H.$$
In this case $\Lambda, \widetilde{\Lambda}$ are said
$CC'$-controlled g-dual pair also.
\end{defn}
The proof of the following is straightforward.
\begin{prop}
If $\Lambda, \widetilde{\Lambda}$ are $CC'$-controlled g-dual pair,
then the following statements are equivalent:
\begin{enumerate}
\item[(i)] $T_{\Lambda} T^*_{\widetilde{\Lambda}}=Id_H$;
\item[(ii)] $T_{\widetilde{\Lambda}} T^*_{\Lambda}=Id_H$;
\item[(iii)] $\langle f, g\rangle=\langle T^*_{\widetilde{\Lambda}}f, T^*_{\Lambda}g\rangle$,  $f, g\in H$.
\end{enumerate}
Also, for every $f\in H$, we have
\begin{equation}\label{dual}
f=\sum_{i\in\Bbb
I}(C^*\Lambda^*_i\Lambda_iC')^{\frac{1}{2}}(C^*\widetilde{\Lambda}^*_i\widetilde{\Lambda}_iC')^{\frac{1}{2}}f.
\end{equation}
\end{prop}
Now, we want to present the \textit{canonical controlled g-dual} by
\eqref{s1} in the case that $C=C'$ and $\Lambda$ is a
$C^2$-controlled g-frame for $H$. Let
$\Gamma_i:=\Lambda_iCS_{C}^{-1}$. Therefore, for each $f\in H$
\begin{equation}\label{d1}
f=\sum_{i\in\Bbb I}C^*\Lambda^*_i\Gamma_i f=\sum_{i\in\Bbb
I}\Gamma^*_i\Lambda_i Cf.
\end{equation}
We show that $\Gamma:=\{\Gamma_i\}_{i\in\Bbb I}$ is a  g-frame for
$H$. Let $f\in H$ and $A_C, B_C$ be the frame bounds of $\Lambda$.
Then
\begin{align*}
\sum_{i\in\Bbb I}\Vert\Gamma_i f\Vert^2&=\sum_{i\in\Bbb I}\langle \Lambda_iCS_{C}^{-1}f, \Lambda_iCS_{C}^{-1}f\rangle\\
&=\sum_{i\in\Bbb I}\langle C^*\Lambda^*_i\Lambda_iCS^{-1}_{C}f, S^{-1}_{C}f\rangle\\
&=\langle f, S^{-1}_{C}f\rangle\\
&\leq\frac{1}{A_C}\Vert f\Vert^2.
\end{align*}
On the other hand,
\begin{align*}
\Vert f\Vert^4&=\langle f, f\rangle^2\\
&=\langle\sum_{i\in\Bbb I}\Gamma^*_i\Lambda_i Cf, f\rangle^2\\
&=\langle\sum_{i\in\Bbb I}\Lambda_i Cf, \Gamma_i f\rangle^2\\
&\leq\big(\sum_{i\in\Bbb I}\Vert \Lambda_i Cf\Vert^2\big)\big(\sum_{i\in\Bbb I}\Vert\Gamma_if\Vert^2\big)\\
&\leq B_C\Vert f\Vert^2\big(\sum_{i\in\Bbb
I}\Vert\Gamma_if\Vert^2\big).
\end{align*}
Finally, we conclude that
$$\frac{1}{B_C}\Vert f\Vert^2\leq\sum_{i\in\Bbb I}\Vert\Gamma_if\Vert^2\leq\frac{1}{A_C}\Vert f\Vert^2.$$
The following theorem shows that the canonical controlled g-dual
gives rise to expand coefficients with the minimal norm.
\begin{thm}
Let $\Lambda=\{\Lambda_i\in\mathcal{B}(H, H_i)\}_{i\in\Bbb I}$ be a
$C^2$-controlled g-frame for $H$ and
$\Gamma_i=\Lambda_iCS_{C}^{-1}$. If $f$ has a representation
$f=\sum_{i\in \Bbb I}C^*\Lambda_i^* g_i$, for some $g_{i}\in H_{i}$.
Then we have
$$\sum_{i\in\Bbb I}\Vert g_i\Vert^2=\sum_{i\in\Bbb I}\Vert\Gamma_i f\Vert^2+\sum_
{i\in\Bbb I}\Vert g_i-\Gamma_i f\Vert^2,\ \ \ f\in H$$
\end{thm}
\begin{proof}
Assume that $f\in H$. We get by \eqref{d1}
\begin{align*}
\sum_{i\in\Bbb I}\Vert\Gamma_i f\Vert^2&=\sum_{i\in\Bbb I}\langle \Gamma_i f, \Lambda_i CS_{C}^{-1}f\rangle\\
&=\sum_{i\in\Bbb I}\langle C^*\Lambda_i^*\Gamma_i f, S_{C}^{-1}f\rangle\\
&=\sum_{i\in\Bbb I}\langle C^*\Lambda_i^*g_i, S_{C}^{-1}f\rangle\\
&=\sum_{i\in\Bbb I}\langle g_i, \Lambda_i CS_{C}^{-1}f\rangle\\
&=\sum_{i\in\Bbb I}\langle g_i, \Gamma_i f\rangle.
\end{align*}
Therefore, $\mbox{Im}\Big(\sum_{i\in\Bbb I}\langle g_i, \Gamma_i
f\rangle\Big)=0$ and the conclusion follows.
\end{proof}
\section{Some equalities and inequalities}
In this section, we extend some known equalities and inequalities
for controlled g-frames. Assume that $\Lambda, \widetilde{\Lambda}$
are $CC'$-controlled g-dual pair and $\Bbb J\subset\Bbb I$. We
define $$S_{\Bbb J}f:=\sum_{i\in\Bbb
J}(C^*\Lambda^*_i\Lambda_iC')^{\frac{1}{2}}(C^*\widetilde{\Lambda}^*_i\widetilde{\Lambda}_iC')^{\frac{1}{2}}f
, \ \ \ \ \ f\in H.$$
 It is clear that $S_{\Bbb J}\in\mathcal{B}(H)$ and $S_{\Bbb J}+S_{\Bbb J^{c}}=Id_H$ where  $\Bbb J^c$ is the complement of $\Bbb J$. Indeed, if $B_1$ and $B_2$ are the bounds of $\Lambda$ and $\widetilde{\Lambda}$ respectively, then
\begin{align*}
\Vert S_{\Bbb J}f\Vert^2&=\Big(\sup_{\Vert g\Vert=1}\vert\langle S_{\Bbb J}f, g\rangle\vert\Big)^2\\
&\leq\Big(\sup_{\Vert g\Vert=1}\sum_{i\in\Bbb J}\big\vert\big\langle(C^*\Lambda^*_i\Lambda_iC')^{\frac{1}{2}}(C^*\widetilde{\Lambda}^*_i\widetilde{\Lambda}_iC')^{\frac{1}{2}}f, g\big\rangle\big\vert\Big)^2\\
&\leq\Big(\sum_{i\in\Bbb I}\Vert(C^*\Lambda^*_i\Lambda_iC')^{\frac{1}{2}}f\Vert^2\Big)\Big(\sup_{\Vert g\Vert=1}\sum_{i\in\Bbb I}\Vert(C^*\widetilde{\Lambda}^*_i\widetilde{\Lambda}_iC')^{\frac{1}{2}}g\Vert^2\Big)\\
&\leq B_1 B_2\Vert f\Vert^2.
\end{align*}
So, $S_{\Bbb J}$ is bounded.
\begin{thm}
If $f\in H$ then,
\begin{eqnarray*}
\sum_{i\in\Bbb J}\langle (C^*\widetilde{\Lambda}^*_i\widetilde{\Lambda}_iC')^{\frac{1}{2}}f, (C^*\Lambda^*_i\Lambda_iC')^{\frac{1}{2}}f\rangle-\Vert S_{\Bbb J}f\Vert^2\\
=\sum_{i\in\Bbb J^c}\overline{\langle
(C^*\widetilde{\Lambda}^*_i\widetilde{\Lambda}_iC')^{\frac{1}{2}}f,
(C^*\Lambda^*_i\Lambda_iC')^{\frac{1}{2}}f\rangle}-\Vert S_{\Bbb
J^c}f\Vert^2
\end{eqnarray*}
\end{thm}
\begin{proof}
 Let $f\in H$. We obtain
 \begin{align*}
\sum_{i\in\Bbb J}\langle (C^*\widetilde{\Lambda}^*_i\widetilde{\Lambda}_iC')^{\frac{1}{2}}f, (C^*\Lambda^*_i\Lambda_iC')^{\frac{1}{2}}f\rangle-\Vert S_{\Bbb J}f\Vert^2&=\langle S_{\Bbb J}f, f\rangle-\Vert S_{\Bbb J}f\Vert^2\\
&=\langle S_{\Bbb J}f,f\rangle-\langle S^{\ast}_{\Bbb J} S_{\Bbb J}f,f\rangle\\
&=\langle(id_{H}-S_{\Bbb J})^{\ast}S_{\Bbb J}f,f\rangle\\
&=\langle S^{\ast}_{\Bbb J^c}(id_{H}-S_{\Bbb J^c})f,f\rangle\\
&=\langle S^{\ast}_{\Bbb J^c}f,f\rangle-\langle S^{\ast}_{\Bbb J^c}S_{\Bbb J^c}f,f\rangle\\
&=\langle f,S_{\Bbb J^c}f\rangle-\langle S_{\Bbb J^c}f,S_{\Bbb J^c}f\rangle.\\
\end{align*}
Now, the proof is completed.
\end{proof}
\begin{cor}\label{cor1}
If $\Lambda$ is a $CC'$-controlled Parseval g-frame for $H$, then
\begin{small}
\begin{align*}
\sum_{i\in\Bbb J}\Vert(C^*\Lambda^*_i\Lambda_iC')f\Vert^2-\Vert \sum_{i\in\Bbb J}(C^*\Lambda^*_i &\Lambda_iC')f\Vert^2\\
&=\sum_{i\in\Bbb J^c}\Vert(C^*\Lambda^*_i\Lambda_iC')f\Vert^2-\Vert
\sum_{i\in\Bbb J^c}(C^*\Lambda^*_i\Lambda_iC')f\Vert^2.
\end{align*}
\end{small}
Moreover,
$$\sum_{i\in\Bbb J}\Vert(C^*\Lambda^*_i\Lambda_iC')f\Vert^2+\Vert \sum_{i\in\Bbb J^c}(C^*\Lambda^*_i\Lambda_iC')f\Vert^2\geq\frac{3}{4}\Vert f\Vert^2.$$
\end{cor}
\begin{proof}
If $f\in H$, we obtain
\begin{align*}
\sum_{i\in\Bbb J}\Vert(C^*\Lambda^*_i\Lambda_iC')f\Vert^2+\Vert
S_{\Bbb J^c}f\Vert^2&=
\big\langle(S_{\Bbb J}+S^2_{\Bbb J^c})f, f\big\rangle\\
&=\big\langle(S_{\Bbb J}+id_H-2S_{\Bbb J}+S^2_{\Bbb J})f, f\big\rangle\\
&=\langle(id_H-S_{\Bbb J}+S_{\Bbb J}^2)f, f\rangle.
\end{align*}
Now, by Lemma \ref{l0} for $a=1$, $b=-1$ and $c=1$ the inequality
holds.
\end{proof}
\begin{cor}\label{cor2}
If $\Lambda$ is a $CC'$-controlled Parseval g-frame for $H$, then
$$0\leq S_{\Bbb J}-S_{\Bbb J}^2\leq\frac{1}{4}Id_H.$$
\end{cor}
\begin{proof}
We have $S_{\Bbb J}S_{\Bbb J^c}=S_{\Bbb J^c}S_{\Bbb J}$. Then $0\leq
S_{\Bbb J}S_{\Bbb J^c}=S_{\Bbb J}-S_{\Bbb J}^2$. Also, by Lemma
\ref{l0}, we get
$$S_{\Bbb J}-S_{\Bbb J}^2\leq\frac{1}{4}id_H.$$
\end{proof}
\begin{thm}\label{t3}
Let $\Lambda$ be a $CC'$-controlled g-frame with $CC'$-controlled
g-frame operator $S_{CC'}$. Suppose that $S^{-\frac{1}{2}}_{CC'}$
commutes with $C, C'$. Then for each $f\in H$,
\begin{small}
$$\sum_{i\in\Bbb J}\Vert S^{-\frac{1}{2}}_{CC'}C^*\Lambda_i^*\Lambda_iC'f\Vert^2+\Vert S^{-\frac{1}{2}}_{CC'}S_{\Bbb J^c}f\Vert^2=\sum_{i\in\Bbb J^c}\Vert S^{-\frac{1}{2}}_{CC'}C^*\Lambda_i^*\Lambda_iC'f\Vert^2+\Vert S^{-\frac{1}{2}}_{CC'}S_{\Bbb J}f\Vert^2.$$
\end{small}
\end{thm}
\begin{proof}
Via Remark \ref{exam} and Corollary \ref{cor1}, if
$\Theta_i:=\Lambda_iS^{-\frac{1}{2}}_{CC'}$, then
\begin{align*}
\sum_{i\in\Bbb J}(C^*\Theta^*_i \Theta_iC')f&=\sum_{i\in\Bbb J}(C^*S_{CC'}^{-\frac{1}{2}}\Lambda^*_i \Lambda_i S_{CC'}^{-\frac{1}{2}}C')f\\
&=\sum_{i\in\Bbb J}(S_{CC'}^{-\frac{1}{2}}C^*\Lambda^*_i \Lambda_i C'S_{CC'}^{-\frac{1}{2}})f\\
&=S_{CC'}^{-\frac{1}{2}}S_{\Bbb J}S_{CC'}^{-\frac{1}{2}}f,
\end{align*}
and also
\begin{small}
\begin{align*}
\sum_{i\in\Bbb J}\Vert(C^*\Theta^*_i\Theta_iC')f\Vert^2-\Vert \sum_{i\in\Bbb J}(C^*\Theta^*_i &\Theta_iC')f\Vert^2=\\
&=\sum_{i\in\Bbb J^c}\Vert(C^*\Theta^*_i\Theta_iC')f\Vert^2-\Vert
\sum_{i\in\Bbb J^c}(C^*\Theta^*_i\Theta_iC')f\Vert^2.
\end{align*}
\end{small}
Now, by replacing $S_{CC'}^{\frac{1}{2}}f$ instead of $f$ in above
formulas, the proof is evident.
\end{proof}
\begin{cor}
Let $\Lambda$ be a $CC'$-controlled g-frame with  $CC'$-controlled
g-frame operator $S_{CC'}$. Suppose that $S^{-\frac{1}{2}}_{CC'}$
commutes with $C, C'$. Then
$$0\leq S_{\Bbb J}-S_{\Bbb J}S^{-1}_{CC'}S_{\Bbb J}\leq\frac{1}{4}S_{CC'}.$$
\end{cor}
\begin{proof}
In the proof of Theorem \ref{t3}, we showed that
$$\sum_{i\in\Bbb J}(C^*\Theta^*_i \Theta_iC')f=S_{CC'}^{-\frac{1}{2}}S_{\Bbb J}S_{CC'}^{-\frac{1}{2}}f.$$
By Corollary \ref{cor2} we get
$$0\leq\sum_{i\in\Bbb J}(C^*\Theta^*_i \Theta_iC')f-\big(\sum_{i\in\Bbb J}(C^*\Theta^*_i \Theta_iC')f\big)^2\leq\frac{1}{4}Id_H.$$
So,
$$0\leq S^{-\frac{1}{2}}_{CC'}(S_{\Bbb J}-S_{\Bbb J}S^{-1}_{CC'}S_{\Bbb J})S^{-\frac{1}{2}}_{CC'}\leq\frac{1}{4}S_{CC'},$$
and it completes the proof.
\end{proof}
\begin{cor}
Let $\Lambda$ be a $CC'$-controlled g-frame with  $CC'$-controlled
g-frame operator $S_{CC'}$ . Suppose that $S^{-\frac{1}{2}}_{CC'}$
commutes  $C, C'$. Then for each  $f\in H$,
$$\sum_{i\in\Bbb J}\Vert S^{-\frac{1}{2}}_{CC'}C^*\Lambda_i^*\Lambda_iC'f\Vert^2+\Vert S^{-\frac{1}{2}}_{CC'}S_{\Bbb J^c}f\Vert^2\geq\frac{3}{4}\Vert S^{-\frac{1}{2}}_{CC'}\Vert^{-1}\Vert f\Vert^2.$$
\end{cor}
\begin{proof}
Applying Theorem \ref{t3} and Corollary \ref{cor1}, we obtain
\begin{align*}
\sum_{i\in\Bbb J}\Vert S^{-\frac{1}{2}}_{CC'}C^*\Lambda_i^*&
\Lambda_iC'f\Vert^2+\Vert S^{-\frac{1}{2}}_{CC'}S_{\Bbb J^c}f\Vert^2=\\
&=\sum_{i\in\Bbb J}\Vert(C^*\Theta^*_i\Theta_iC')S^{\frac{1}{2}}_{CC'}
f\Vert^2-\Vert \sum_{i\in\Bbb J}(C^*\Theta^*_i\Theta_iC')S^{\frac{1}{2}}_{CC'}f\Vert^2\\
&\geq\frac{3}{4}\Vert S^{\frac{1}{2}}_{CC'}f\Vert^2\\
&=\frac{3}{4}\langle S_{CC'}f, f\rangle\\
&\geq\frac{3}{4}\Vert S^{-1}_{CC'}\Vert^{-1}\Vert f\Vert^2.
\end{align*}
\end{proof}
\begin{thm}
Let $\Lambda$ be a Parseval $CC'$-controlled g-frame for $H$. Then
\begin{enumerate}
\item [(i)] $0\leq S_{\Bbb J}-S_{\Bbb J}^2\leq\dfrac{1}{4}id_H$.
\item [(ii)] $\dfrac{1}{2}id_H\leq S_{\Bbb J}^2+S_{\Bbb J^c}^2\leq\dfrac{3}{2}id_H$.
\end{enumerate}
\end{thm}
\begin{proof}
(i). Since $S_{\Bbb J}+S_{\Bbb J^c}=id_H$, then $S_{\Bbb J}S_{\Bbb
J^c}+S^2_{\Bbb J^c}=S_{\Bbb J^c}$. Thus
$$S_{\Bbb J}S_{\Bbb J^c}=S_{\Bbb J^c}-S^2_{\Bbb J^c}=S_{\Bbb J^c}(id_H-S_{\Bbb J^c})=S_{\Bbb J^c}S_{\Bbb J}.$$
But, $\Lambda$ is a Parseval $CC'$-controlled g-frame, so $0\leq
S_{\Bbb J}S_{\Bbb J^c}=S_{\Bbb J}-S_{\Bbb J}^2$. On the other hand,
by Lemma \ref{l2}, we get
$$S_{\Bbb J}-S_{\Bbb J}^2\leq\frac{1}{4}id_H.$$
(ii). Since $S_{\Bbb J}S_{\Bbb J^c}=S_{\Bbb J^c}S_{\Bbb J}$, so by
Lemma \ref{l2}
$$S_{\Bbb J}^2+S_{\Bbb J^c}^2=id_H-2S_{\Bbb J}S_{\Bbb J^c}=2S_{\Bbb J}^2-2S_{\Bbb J}+id_H\geq\frac{1}{2}id_H.$$
and we get the right inequality  by Lemma \ref{l2} and $0\leq
S_{\Bbb J}-S_{\Bbb J}^2$,
$$S_{\Bbb J}^2+S_{\Bbb J^c}^2\leq id_H+2S_{\Bbb J}-2S_{\Bbb J}^2\leq\frac{3}{2}id_H.$$
\end{proof}

 ----------------------------------------------------------------

\end{document}